\documentclass[reqno,final]{amsart}
\usepackage{natbib}  
\usepackage{fancyhdr} 
\usepackage{color} 
\usepackage{hyperref} 
\usepackage{graphicx} 
\usepackage{epstopdf}
\usepackage{amsmath}
\usepackage{pstricks}
\usepackage{amssymb}
\usepackage{cite}

\definecolor{aleacolor}{rgb}{0.16,0.59,0.78}

\hypersetup{
breaklinks,
colorlinks=true,
linkcolor=aleacolor,
urlcolor=aleacolor,
citecolor=aleacolor}

\renewcommand{\headheight}{11pt}
\renewcommand{\cite}{\citet}

\theoremstyle{plain}
\newtheorem{theorem}{Theorem}[section]                                          
\newtheorem{proposition}[theorem]{Proposition}                          
\newtheorem{lemma}[theorem]{Lemma}
\newtheorem{corollary}[theorem]{Corollary}

\theoremstyle{definition}
\newtheorem{definition}[theorem]{Definition}
\theoremstyle{remark}
\newtheorem{remark}[theorem]{Remark}

\makeatletter \@addtoreset{equation}{section} \makeatother
\renewcommand\theequation{\thesection.\arabic{equation}}
\renewcommand\thefigure{\thesection.\arabic{figure}}
\renewcommand\thetable{\thesection.\arabic{table}}

\begin{document}
  
\title{Random walk on a randomly oriented honeycomb lattice} 



\author{Gianluca Bosi}
\author{Massimo Campanino}

\address{Universit\`a degli Studi di Bologna\newline
Piazza di Porta San Donato 5, 40126\newline
Bologna, Italy.}

\address{Universit\`a degli Studi di Bologna\newline
Piazza di Porta San Donato 5, 40126\newline
Bologna, Italy.}

\email{gianluca.bosi4@unibo.it, massimo.campanino@unibo.it}




\begin{abstract}
We study the recurrence behaviour of random walks on partially oriented
honeycomb lattices. The vertical edges are undirected while the orientation
of the horizontal edges is random:
depending on their distribution, we prove a.s. transience in some cases, and a.s.
recurrence in other ones. The results extend those
obtained for the partially oriented square grid lattices  (\cite{ref1}, \cite{ref2}). 
\end{abstract}

\maketitle

\section{Introduction}
\subsection{Motivation}
The behaviour of random walks on oriented two-dimensional graphs
has been object in the last years of several works. In particular new methods have
to be devised to settle the question of their recurrence or transience. The main
result of \cite{ref1} concerns the random walk on the two-dimensional
square lattice where the vertical edges are undirected while the edges
on each horizontal line are all oriented randomly and independently to the right
or to the left with equal probability. The result is that this random walk is transient
almost surely with respect to the environment; in \cite{ref2} the study is generalized to a more general class of environments, in particular
showing that the random walk is recurrent provided that the lines are oriented
through periodic functions. In the last decade many authors have investigated this and
related models: just to cite some of them, in \cite{guillotin3} a functional limit theorem for the random walk is proved, in \cite{guillotin4} a local limit theorem is established, and in \cite{guillotin5} the range of the walk is analyzed; in \cite{pene} the author considers the case where the lines are oriented by a sequence of stationary random variables, while in \cite{pene2} the model is generalized to one where the probability of staying on a line is non-constant. A common characteristic of the considered random walks is that
one can split them, apart from a time change, into a ``horizontal" and ``vertical"
component, where the latter is independent from the former. Moreover all these studies deal with the square grid lattice, 
and one can ask if these recurrence properties still hold as the geometry of the underlying two-dimensional 
graph changes: motivated by this question, in the present work we consider
the random walk on the honeycomb randomly directed lattice. Here the ``vertical
motion" obtained from the splitting is no longer independent from the ``horizontal" one: as a result, the subsequent steps of the ``vertical motion" have a markovian
dependency. For this reason, while some of the techniques used in \cite{ref1} and \cite{ref2} can be properly extended and
adopted in our study, in several parts of the proof we need to develop new ideas
and techniques. While we consider just a specific example also for reasons of simplicity, 
we expect  in the future to be able to extend
 our approach to a more general setting.
\subsection{Notation and results}
Let $\mathbf{L}:=(\mathbb{Z}^2,E)$ be the square grid lattice, i.e. $E$ is the set of nearest neighbours in $\mathbb{Z}^2$. The honeycomb lattice can be defined as the sub-graph obtained from $\mathbf{L}$ by eliminating the following set of edges (see figure \ref{figure::honeycomb})
\begin{align*}
&\{((2j,2k),(2j,2k+ 1))|\,j,k\in\mathbb{Z}\}\\
\cup&\{((2j,2k+1),(2j,2k))|\,j,k\in\mathbb{Z}\}\\
\cup&\{((2j+1,2k+1),(2j+1,2k+2))|\,j,k\in\mathbb{Z}\}\\
\cup&\{((2j+1,2k+2),(2j+1,2k+1))|\,j,k\in\mathbb{Z}\}.
\end{align*}
Then we can define partially oriented versions of the lattice by imposing a certain orientation on the horizontal edges (this can be done either deterministically or randomly), while keeping the vertical edges unoriented. Our work is devoted to the study of the recurrence (transience) behaviour of a simple random walk $(M_n)_{n\geq 0}$ on such oriented lattices.
Let $(\Omega, \mathcal{F}, \mathbb{P})$ be a probability space. The first result we prove is the following:
\begin{figure} 
\begin{center}
\includegraphics[width=9cm]{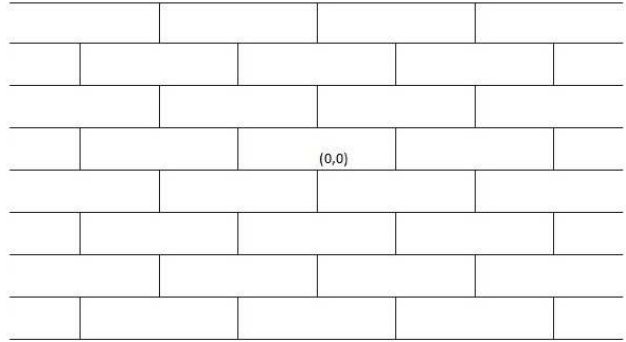}
\caption{The honeycomb lattice}
\label{figure::honeycomb}
\end{center}
\end{figure}

\begin{theorem}\label{Th:: RANDOM}
Let $(\epsilon_y)_{y\in\mathcal{Z}}$ be a i.i.d. family of $\{-1,1\}$-valued Rademacher random variables and denote by $\mathbf{H}_\epsilon$ the honeycomb lattice oriented randomly as follows: if $\epsilon_y=1$ there is only a right-directed edge between $(x,y)$ and $(x+1,y)$, $\forall x\in\mathbb{Z}$; if $\epsilon_y=-1$, only a left-directed one.
Then the random walk $(M_n)_{n\geq 0}$ on $\mathbf{H}_\epsilon$ is a.s. transient.
\end{theorem}

Since the transience behaviour is caused by the presence and the size of fluctuations in the orientations, we impose periodic orientations and prove the following result.

\begin{theorem}\label{Th:: PERIODIC}
Let $Q>1$ be an even integer and, given a $Q$-periodic function $f:\mathbb{Z}\longrightarrow \{-1,1\}$ such that $\sum_{k=0}^{Q-1}f(k)=0$, consider the oriented honeycomb lattice $\mathbf{H}_{f}$ whose horizontal edges are oriented according to the value of $f$: that is, if $f(k)=1$ then the edges with ordinate $k$ are right-directed, otherwise they are left-directed.
Then the random walk $(M_n)_{n\geq 0}$ on $\mathbf{H}_f$ is recurrent.
\end{theorem}

Finally, we consider the case of horizontal orientations prescribed by a random perturbation of a periodic function.
\begin{theorem}\label{Th:: PERTURBED}
Let $(\epsilon_y)_{y\in\mathbb{Z}}$ and $f$ as above and define for every $y\in\mathbb{Z}$
$$\overline\epsilon_y:=(1-\lambda_y)f(y)+\lambda_y\epsilon_y$$
where $\lambda=(\lambda_y)_{y\in\mathbb{Z}}$ is a $\{0,1\}$-valued sequence of independent r.v, independent of $\epsilon$ s.t. $$\mathbb{P}(\lambda_y=1)=\frac{c}{|y|^\beta}.$$
Denote by $\mathbf{H}_{\overline \epsilon,\lambda}$ the honeycomb lattice with such random orientations. Then

(i) If $\beta<1$, the random walk $M_n$ on $\mathbf{H}_{\overline \epsilon,\lambda}$ is ($\epsilon_y,\lambda_y$)-a.s. transient.

(ii) If $\beta>1$, the random walk $M_n$ on $\mathbf{H}_{\overline \epsilon,\lambda}$ is ($\epsilon_y,\lambda_y$)-a.s. recurrent.
\end{theorem}

\section{Technical preliminaries}
\subsection{Decomposition of the random walk}
Following \cite{ref1}, we decompose the random walk into two components that, if sampled on a particular sequence of random times, have the same recurrence behaviour of $(M_n)$.

We begin with the following observation.
Let $\xi$ be a random variable with geometric distribution of parameter $\frac{1}{2}$, and consider the event of an even or null outcome $A:=\bigcup_{m\in\mathbb{N}\cup\{0\}}\{\xi =2m\}$. We have

$$\mathbb{P}(A)=\sum_{m= 0}^\infty \mathbb{P}(\xi=2m)=\frac{1}{2}\sum_{m= 0}^\infty \left(\frac{1}{2}\right)^{2m}=\frac{2}{3}.$$
Obviously $\mathbb{P}(A^c)=\frac{1}{3}$. Now, if we interpret $\xi$ as the absolute value of the first horizontal displacement of $(M_n)$ (i.e. the number of subsequent horizontal steps before the first vertical one), we immediately see that after an odd outcome the random walk will perform a vertical down-directed step, otherwise a up-directed one. 
With this observation in mind, we give the following definition.
\begin{definition}
The \textit{vertical skeleton} of $(M_n)_{n\geq0}$ is the Markov process $(Y_n, \nu_n)_{n\geq0}$ with values in $\mathbb{Z}\times\{-1,1\}$ defined by the following transition probabilities:
\begin{align*}
& p_{(y,1), (y+1,1)}=p_{(y,-1),(y-1,-1)}=\frac{1}{3}\\
& p_{(y,1),(y-1,-1)}=p_{(y,-1)(y+1,1)}=\frac{2}{3}
\end{align*}
for any $y\in\mathbb{Z},$ and $\mathbb{P}((Y_0,\nu_0)=(0,1))=1.$ 
\end{definition}
The first component of the vertical skeleton represents the projection  on the $y$-axis of the position of the random walk, seen at the times of successive vertical steps, while the second component represents the speed -or direction- of the last step. 
\begin{remark}
By definition, the skeleton random walk $(Y_{k},\nu_k)$ satisfies $$Y_k=\sum_{i=1}^k \nu_i,$$ and the transition matrix of the Markov chain $(\nu_k)_{k\geq 0}$ has the following form:
\[ \pi_{\nu}=\left( \begin{array}{cc}
q & 1-q \\
1-q & q \end{array} \right)\text{,\,\, with } q=\frac{1}{3}.\]
\end{remark}
Now we define the occupation measure of the vertical skeleton and the embedded random walk.
\begin{definition}
Let $n\geq 0$. Let for every $y\in\mathbb{Z}$
$$\eta_n(y,\pm1):=\sum_{k=0}^n\mathbf{1}_{\{(Y_k,\nu_k)=(y,\pm1)\}}.$$
We define the total \textit{occupation measure}
$\eta_n$ at $y$ by
$$\eta_n(y):=\eta_n(y,1)+\eta_{n}(y,-1).$$
\end{definition}

\begin{definition}
Let $(\xi_i)^{(y)}_{i\geq0, \, y\in \mathbb{Z}}$  be a family of i.i.d. geometric random variables with parameter $p=\frac{1}{2}$, defined on $(\Omega, \mathcal{A}, \mathbb{P})$. We call \textit{embedded random walk} the process $(X_n)_{n\geq0}$ defined by\\
$$X_n:=\sum_{y\in\mathbb{Z}}\epsilon_y\sum_{i=1}^{\eta_{n-1}(y)}\xi_i^{(y)}$$
with the convention that $\sum_i$ vanishes whenever $\eta_{n-1}(y)=0$.
\end{definition}
$X_n$ represents the abscissa of the random walk $(M_n)_{n \geq 0}$ immediately after the $n$-th vertical movement has been performed.

Let $T_n=n+\sum_{y\in\mathbb{Z}}\sum_{i=1}^{\eta_{n-1}(y)}\xi_i^{(y)}$ be the time just after the random walk $(M_n)$ has performed its $n$-th vertical move. Then it's straightforward to see that $M_{T_n}=(X_n,Y_n),$ where $(Y_n)$ is the first component of the vertical skeleton $(Y_n, \nu_n).$ Now denote by $\sigma_n$ the sequence of consecutive returns to $0$ of $(Y_n)$ (by lemma \ref{Y_n_behaviuour} below it follows that $\sigma_n<\infty$ almost surely $\forall n$.) Obviously, $M_{T_{\sigma_n}}=(X_{\sigma_n},0)$. 

Let $\mathcal{F}_n:=\sigma(\nu_k, k\leq n)$, $\mathcal{G}:=\sigma(\epsilon_y,y\in\mathbb{Z})$ and $\mathcal{F}\vee\mathcal{G}=\sigma(\mathcal{F}\cup\mathcal{G})$, where $\mathcal{F}=\vee_n \mathcal{F}_n$. We shall need the following lemmas.
\begin{lemma}\label{Y_n_behaviuour}
As $n\to\infty$
$$\mathbb{P}_0(Y_{2n}=0)\sim\frac{C}{\sqrt{n}},$$
with $C>0$. In particular, $Y_n$ is recurrent.
\end{lemma}
\begin{proof}
Since $Y_{2n}:=\sum_{k=1}^{2n} \nu_k$, the result follows by the local limit theorem for ergodic Markov chain with finite state space (see \cite{refkolmogorov}).
\end{proof}
\begin{lemma}\label{keylemma}[\cite{ref1}, lemma 2.3]
$(X_{\sigma_n})_{n\geq 0}$ is transient $\implies (M_n)_{n\geq 0} $ is transient, i.e.
\begin{equation*}
\sum_{n=0}^{\infty} \mathbb{P}_0(X_{\sigma_n}=0\mid\mathcal{F}\vee\mathcal{G})<\infty \text{ a.s } \implies \sum_{l=0}^{\infty} \mathbb{P}_0(M_l=(0,0)\mid\mathcal{G})<\infty \text{ a.s }.
\end{equation*}
\end{lemma}


\subsection{Characteristic function of the embedded random walk}
Let $n\in\mathbb{N}, y\in\mathbb{Z}$ and define $$m_{n,o}(y):=\sum_{k=0}^n\mathbf{1}_{\{Y_k=y, \nu_k=\nu_{k+1}\}},$$ $$m_{n,e}(y):=\sum_{k=0}^n\mathbf{1}_{\{Y_k=y, \nu_k\ne\nu_{k+1}\}}.$$ 
They satisfy $m_{n,o}(y)+m_{n,e}(y)=\eta_{n}(y).$ 
So we can decompose the embedded random walk $X_n$ as follows
$$X_n=\sum_{y\in\mathbb{Z}}\epsilon_y\left(\sum_{i=1}^{m_{n-1,o}(y)}\xi_{i,o}^{(y)}+\sum_{i=1}^{m_{n-1,e}(y)}\xi_{i,e}^{(y)}\right),$$
where $\xi_{i,o}^{(y)}$ and $\xi_{i,e}^{(y)}$ are two independent families of i.i.d. random variables having, respectively, the law of a geometric random variable taking only odd integer values and only null or even integer values; precisely, it is easy to see that
$\mathbb{P}(\xi_{i,o}^{(y)}=2k+1)=\mathbb{P}(\xi_{i,e}^{(y)}=2k)=3\left(\frac{1}{2}\right)^{2k+2}
$  for every $k\in\mathbb{N}.$ In the present work, we shall say that a random variable is \textit{even geometric} if it has the same law of $\xi_{1,e}^{(0)}$, and \textit{odd geometric} if it has the same law of $(\xi_{1,o}^{(0)})$.
Their characteristic function are, respectively, $$\chi_o(\theta):=\mathbb{E}(\exp(i\theta \xi_{i,o}^{(y)}))=\frac{3e^{i\theta}}{4-e^{2i\theta}}$$
and
$
\chi_e(\theta):=\mathbb{E}(\exp(i\theta \xi_{i,e}^{(y)}))=e^{-i\theta}\chi_o(\theta).$
Observe that $r_e(\theta):=|\chi_e(\theta)|=|e^{-i\theta}\chi_o(\theta)|=|\chi_o(\theta)|=r_o(\theta)$. Moreover, note that $r_o(\theta)$ is an even function.

\begin{lemma}\label{genfunction1}
The characteristic function of $X_n$ is
$$\mathbb{E}(\exp(i\theta X_n))=\mathbb{E}\left(\prod_{y\in \mathbb{Z}}\chi_o(\theta\epsilon_y)^{m_{n-1,o}(y)}\chi_e(\theta\epsilon_y)^{m_{n-1,e}(y)}\right).$$
\end{lemma}
\begin{proof}
We have
\begin{align*}
\mathbb{E}(\exp(i\theta X_n))=&\mathbb{E}\left(\mathbb{E}\left(\exp(i\theta \sum_{y\in\mathbb{Z}}\epsilon_y\sum_{i=1}^{\eta_{n-1}(y)}\xi_i^{(y)})\mid \mathcal{F}_n\vee\mathcal{G}\right)\right)\\
=& \mathbb{E}\left(\mathbb{E}\left(\prod_{y\in \mathbb{Z}}\exp(i\theta \epsilon_y\sum_{i=1}^{\eta_{n-1}(y)}\xi_i^{(y)})\mid \mathcal{F}_n\vee\mathcal{G}\right)\right)\\
=& \mathbb{E}\left(\mathbb{E}\left(\prod_{y\in \mathbb{Z}}\exp(i\theta \epsilon_y(\sum_{i=1}^{m_{n-1,o}^{(y)}}\xi_{i,o}^{(y)}+\sum_{i=1}^{m_{n-1,e}^{(y)}}\xi_{i,e}^{(y)}))\mid \mathcal{F}_n\vee\mathcal{G}\right)\right)\\
=&\mathbb{E}\left(\prod_{y\in \mathbb{Z}}\chi_o(\theta\epsilon_y)^{m_{n-1,o}^{(y)}}\chi_e(\theta\epsilon_y)^{m_{n-1,e}^{(y)}}\right)
\end{align*}
\end{proof}
\section{Proofs}
\subsection{The random walk on the $\mathbf{H}_{\epsilon}$ lattice}
This section is devoted to the proof of theorem \ref{Th:: RANDOM}.
Let $m_o$ and $m_e$ be, respectively, the mean of an odd geometric and of an even geometric random variable. Define $\hat\eta_{2n-1}(y)=m_o m_{2n-1, o}^{(y)} + m_e m_{2n-1, e}^{(y)} .$
As in \cite{ref1} we define, for $n\geq 0$, the following families of events:
\begin{align*}
& A_{n,1}:=\{\max_{0\leq k\leq 2n}|Y_k|<n^{\frac{1}{2} + \delta_1}\}, \,\,\delta_1>0 \\
& A_{n,2}:=\{\max_{y\in\mathbb{Z}}\eta_{2n-1}(y)<n^{\frac{1}{2} + \delta_2}\}, \,\,\delta_2>0 \\
& A_n:=A_{n,1}\cap A_{n,2} \\
& B_n:=A_n \cap \{|\sum_{y\in\mathbb{Z}}\epsilon_y \hat \eta_{2n-1}(y)|>n^{\frac{1}{2} + \delta_3}\}, \,\,\delta_3>0
\end{align*}
where $\delta_1$, $\delta_2$ and $\delta_3$ will be chosen later. Observe that, for every $n$, $A_n\in\mathcal{F}_{2n}$ and $B_n\subset A_n, \,\,B_n\in\mathcal{F}_{2n}\vee\mathcal{G}$.
Thus, we have
$p_n=p_{n,1}+p_{n,2}+p_{n,3},$
where
\begin{align*}
& p_n=\mathbb{P}(X_{2n}=0, Y_{2n}=0)\\
& p_{n,1}=\mathbb{P}(X_{2n}=0, Y_{2n}=0, B_n)\\
& p_{n,2}=\mathbb{P}(X_{2n}=0, Y_{2n}=0, A_n\slash B_n)\\
& p_{n,3}=\mathbb{P}(X_{2n}=0, Y_{2n}=0, A_n^c).
\end{align*}
In order to prove transience, we will provide estimates of -respectively- $p_{n,1}$, $p_{n,2}$ and $p_{n,3}$, from which we will deduce that $\sum_{n\geq 0} p_{n}$ is convergent. Then the result will follow at once thanks to the following lemma.
\begin{lemma}
If $\sum_{n\geq 0} p_{n}<\infty$, then $(M_n)_{n\geq 0}$ is transient.
\end{lemma}
\begin{proof}
From the trivial majorization $$\sum_{n\geq 0}\mathbb{P}(X_{\sigma_n}=0)=\sum_{n\geq 0} \mathbb{P}(X_{\sigma_n}=0, Y_{\sigma_n}=0) \leq \sum_{n\geq 0} \mathbb{P}(X_{2n}=0, Y_{2n}=0),$$
we deduce that $\sum_{n\geq 0} \mathbb{P}(X_{\sigma_n}=0)<\infty$ and hence also $\sum_{n\geq 0}\mathbb{P}(X_{\sigma_n}=0 \mid \mathcal{F}_n\vee\mathcal{G})<\infty \text{ a.s.}$
By lemma \ref{keylemma}, this implies the a.s. transience of $(M_n)_{n\geq 0}$.
\end{proof}

\subsubsection{Estimate of $p_{n,1}$} 
Define $$N_{o}^+:=\sum_{k=1}^{2n}\mathbf{1}_{\{\epsilon_{Y_k}=1\}}\mathbf{1}_{\{\nu_k=\nu_{k+1}\}},$$
$$N_{e}^+:=\sum_{k=1}^{2n}\mathbf{1}_{\{\epsilon_{Y_k}=1\}}\mathbf{1}_{\{\nu_k\ne\nu_{k+1}\}},$$ $$N_o^-:=\sum_{k=1}^{2n}\mathbf{1}_{\{\epsilon_{Y_k}=-1\}}\mathbf{1}_{\{\nu_k=\nu_{k+1}\}},$$ 
$$N_{e}^-:=\sum_{k=1}^{2n}\mathbf{1}_{\{\epsilon_{Y_k}=-1\}}\mathbf{1}_{\{\nu_k\ne\nu_{k+1}\}},$$
and
$$\Delta_{n,o}:=N_o^+-N_o^-,$$
$$\Delta_{n,e}:=N_{e}^+-N_{e}^-,$$
$$\Sigma_{n,o}:=N_o^++N_o^-,$$
$$\Sigma_{n,e}:=N_{e}^++N_{e}^-,$$
Observe that
$$m_o\Delta_{n,o}+m_e\Delta_{n,e}=\sum_{y\in\mathbb{Z}} \epsilon_y\hat \eta_{2n-1}(y)$$ and $\Sigma_{n,o}+\Sigma_{n,e} =2n.$ Let $(\xi_{k,o})_{k\geq 1}$ and $(\xi_{k,e})_{k\geq 1}$ be two families of, respectively, odd geometric and even geometric independent random variables (the families are independent also to each other).
Moreover let $\xi_o$ and $\xi_e$ be respectively a odd geometric and a even geometric random variable, and define  $s_o^2=\sigma^2(\xi_o)$, $s_e=\sigma^2(\xi_e)$.

\begin{lemma}\label{CharX_nExpansion}
We have
$$\mathbb{E}(\exp(t X_{2n})\mid\mathcal{F}_{2n}\vee \mathcal{G})=\exp\left(t(m_o\Delta_{n,o} + m_{e}\Delta_{n,e}) +\frac{t^2}{2}\left(s_o^2\Sigma_{n,o} +  s_{e}^2\Sigma_{n,e}\right)+\mathcal{O}(t^3 n)\right).$$
\end{lemma}
\begin{proof}
Consider the generating function $\phi_o(t)=\mathbb{E}(\exp(t\xi_o))$, defined in $t\in]-\infty,\ln2[$, the largest domain in which $\phi_o(t)<\infty$. 
We have
\begin{align*}
\phi_o(t)=\sum_{k\geq 0} \mathbb{P}(\xi_o=k)e^{tk}=& \sum_{k\geq 0} \mathbb{P}(\xi_o=k)\left(1+kt+\frac{(kt)^2}{2}+\mathcal{O}(t^3)\right)\\
=& 1+\mathbb{E}(t\xi_o)+\frac{1}{2}\mathbb{E}\left(\left(t\xi_o\right)^2\right)+\mathcal{O}(t^3)\\
=& \exp\left(tm_o+t^2\frac{s_o^2}{2}+\mathcal{O}(t^3)\right)
\end{align*}
Analogously, we define $\phi_{e}(t)$ to be the generating function of $\xi_e$, and observe that its generating function behaves as
$\phi_{e}(t)=\exp\left(tm_{e}+t^2\frac{s_{e}^2}{2}+\mathcal{O}(t^3)\right).$
Note that also $\phi_{e}(t)$ is finite if and only if $t\in]-\infty,\ln2[$.
Finally, putting all together, we have  
\begin{align*}
\mathbb{E}(\exp(tX_{2n})\mid \mathcal{F}_{2n}\vee\mathcal{G})
=&\mathbb{E}\left(\exp(t\sum_{k=1}^{2n}\mathbf{1}_{\{\epsilon_{Y_k}=1\}}\xi_k -t\sum_{k=1}^{2n}\mathbf{1}_{\{\epsilon_{Y_k}=-1\}}\xi_k)\mid \mathcal{F}_{2n}\vee\mathcal{G}\right)\\
=&\phi_o(t)^{N_{o}^+}\phi_{e}(t)^{N_{e}^+}\phi_o(-t)^{N_{o}^-} \phi_{e}(-t)^{N_{e}^-}\\
=& \exp\left(t(m_o\Delta_{n,o} + m_{e}\Delta_{n,e}) +\frac{t^2}{2}(s_o^2\Sigma_{n,o} +  s_{e}^2\Sigma_{n,e})+\mathcal{O}(t^3 n)\right).
\end{align*}
\end{proof}
\begin{proposition}\label{estimate-pn1}
For large $n$, on the set $B_n$, we have
$$\mathbb{P}(X_{2n}=0\mid \mathcal{F}_{2n} \vee \mathcal{G})=\mathcal{O}(\exp{(-n^{\delta'})})$$
for any $\delta'\in ]0, 2\delta_3[$.
\end{proposition}
\begin{proof}
Using Markov inequality and lemma \ref{CharX_nExpansion}, we have for $t<0,$
\begin{align*}
\mathbb{P}(X_{2n}=0 \mid \mathcal{F}_{2n} \vee \mathcal{G})\leq & \mathbb{P}(X_{2n}\leq 0 \mid \mathcal{F}_{2n} \vee \mathcal{G})\\
= & \mathbb{P}(tX_{2n}\geq 0 \mid \mathcal{F}_{2n} \vee \mathcal{G})\\
= & \mathbb{P}(\exp(tX_{2n})\geq 1 \mid \mathcal{F}_{2n} \vee \mathcal{G})\\
\leq& \mathbb{E}(\exp(tX_{2n})\mid \mathcal{F}_{2n} \vee \mathcal{G})
\end{align*}
For $0<t<\ln 2$, we obtain analogously the same bound
\begin{align*}
\mathbb{P}(X_{2n}=0 \mid \mathcal{F}_{2n} \vee \mathcal{G})\leq & \mathbb{P}(X_{2n}\geq 0 \mid \mathcal{F}_{2n} \vee \mathcal{G})
\leq \mathbb{E}(\exp(tX_{2n})\mid \mathcal{F}_{2n} \vee \mathcal{G})\\
= & \exp\left(t(m_o\Delta_{n,o} + m_{e}\Delta_{n,e}) +\frac{t^2}{2}\left(s_o^2\Sigma_{n,1} +  s_{e}^2\Sigma_{n,-1}\right)+\mathcal{O}(t^3 n)\right)\\
\leq&\exp(t(m_o\Delta_{n,o} + m_{e}\Delta_{n,e})
+t^2s^2n+O(t^3 n)),
\end{align*}
where $s:=\max\{s_o,s_{e}\}$. Then, for the case $m_o\Delta_{n,o} + m_{e}\Delta_{n,e}>n^{\frac{1}{2}+\delta_3}$, we choose $t=-\frac{n^{\delta_3 -\frac{1}{2}}}{2s^2}$ and get
\begin{align*}
\mathbb{P}(X_{2n}=0 \mid \mathcal{F}_{2n} \vee \mathcal{G})\leq & \exp\left(-\frac{n^{2\delta_3} }{2s^2}+ \frac{n^{2\delta_3-1}s^2n}{4s^4}+\mathcal{O}((n^{\delta_3 -\frac{1}{2}})^3 n)\right)\\
\leq & \exp\left(-\frac{n^{2\delta_3}}{4s^2}+\mathcal{O}(n^{3\delta_3 -\frac{1}{2}})\right)
\end{align*}

Finally, for the case $m_o\Delta_{n,o} + m_{e}\Delta_{n,e}<-n^{\frac{1}{2}+\delta_3}$ we choose $t=\frac{n^{\delta_3 -\frac{1}{2}}}{2s^2}$ and get exactly the same bound.
\end{proof}

\begin{corollary}
$$\sum_{n\in \mathbb{N}} p_{n,1}<\infty.$$
\end{corollary}
\begin{proof}
Observe that
\begin{equation}\label{eq::001}
\mathbb{P}(X_{2n}=0, Y_{2n}=0, B_n \mid \mathcal{F}_{2n}\vee \mathcal{G})\leq \mathbf{1}_{B_n}\mathbb{P}(X_{2n}=0\mid \mathcal{F}_{2n}\vee\mathcal{G}).
\end{equation}
In proposition \ref{estimate-pn1} we proved that, on $B_n$, $\mathbb{P}(X_{2n}=0\mid \mathcal{F}_{2n} \vee \mathcal{G})=\mathcal{O}(\exp{(-n^{\delta'})})$ for $\delta^{'}\in ]0,2\delta_3[.$ Thus, taking expectations on both sides of (\ref{eq::001}) we obtain
$$p_{n,1}\leq \mathbb{E}(\mathcal{O}(\exp{(-n^{\delta'})})\mathbf{1}_{B_n})= \mathcal{O}(\exp{(-n^{\delta'})})\mathbb{E}(\mathbf{1}_{B_n})\leq \mathcal{O}(\exp{(-n^{\delta'})}).$$
Thus, $p_{n,1}$ is summable.
\end{proof}

\subsubsection{Estimate of $p_{n,2}$}
\begin{lemma}\label{Xn-lnn/n}
We have
$$\mathbb{P}(X_{2n}=0 \mid \mathcal{F}_{2n} \vee \mathcal{G})=\mathcal{O}\left(\frac{1}{\sqrt{n}}\right).$$

\end{lemma}
\begin{proof}
In lemma \ref{genfunction1} we saw that the conditional characteristic function of $X_{2n}$ w.r.t. $\mathcal{F}_{2n} \vee \mathcal{G}$ takes on the following form:
$$\chi(\theta):=\mathbb{E}(\exp(i\theta X_{2n})\mid \mathcal{F}_{2n} \vee \mathcal{G})=\prod_{y\in \mathbb{Z}}\chi_o(\theta\epsilon_y)^{m_{2n-1,o}^{(y)}}\chi_e(\theta\epsilon_y)^{m_{2n-1,e}^{(y)}}.$$
We have, by the inversion formula
$$\mathbb{P}(X_{2n}=0 \mid \mathcal{F}_{2n} \vee \mathcal{G})=\frac{1}{2\pi}\int_{-\pi}^\pi \chi(\theta) d\theta.$$
Define $r(\theta):=|\chi_e(\theta)|=|\chi_o(\theta)|=\frac{3}{\sqrt{17-8cos(2\theta)}}$. Thus
\begin{align*}
\frac{1}{2\pi}\int_{-\pi}^\pi \chi(\theta) d\theta\leq &\frac{1}{2\pi}\int_{-\pi}^\pi \prod_{y\in \mathbb{Z}}|\chi_o(\theta\epsilon_y)|^{m_{2n-1,o}^{(y)}}|\chi_e(\theta\epsilon_y)|^{m_{2n-1,e}^{(y)}}d\theta\\ 
=&\frac{1}{2\pi}\int_{-\pi}^\pi \prod_{y\in \mathbb{Z}}|\chi_o(\theta\epsilon_y)|^{\eta_{2n-1}(y)}d\theta\\
=&\frac{1}{2\pi}\int_{-\pi}^\pi  r(\theta)^{\sum_{y\in\mathbb{Z}}\eta_{2n-1}(y)} d\theta=\frac{1}{2\pi}\int_{-\pi}^\pi r(\theta)^{2n} d\theta.
\end{align*}
Now we use the parity of $r(\theta)$ and the fact that $r(\theta)<1$ in $\theta\in]0,\pi[\cup]\pi,2\pi[$ to bound with $K<1$ the function $r(\theta)$ in the interval $[\frac{\pi}{4},\frac{3}{4}\pi] \cup [-\frac{\pi}{4},-\frac{3}{4}\pi].$ We obtain
\begin{align*}
\mathbb{P}(X_{2n}=0 \mid \mathcal{F}_{2n} \vee \mathcal{G})\leq & \frac{1}{\pi}\int_{0}^{\frac{\pi}{4}} r(\theta)^{2n} d\theta +\frac{1}{\pi}\int_{\pi}^{\frac{5\pi}{4}} r(\theta)^{2n} d\theta + \mathcal{O}(K^{2n})\\
=& \frac{2}{\pi}\int_{0}^{\frac{\pi}{4}} r(\theta)^{2n} d\theta + \mathcal{O}(K^{2n}).
\end{align*}
Now, we have $r(\theta)=1-\frac{8}{9}\theta^2+\mathcal{O}(\theta^3)$ and so for large $n$ 
$$\int_{0}^{\frac{\pi}{4}} r(\theta)^{2n} d\theta \sim \int_0^{\frac{\pi}{4}} \left(e^{-\frac{8}{9}\theta^2}\right)^{2n}d\theta\sim \int_0^\infty \left(e^{-\frac{16}{9}n\theta^2}\right)d\theta\sim \frac{c}{\sqrt{n}},$$
with $c=\sqrt{\frac{9\pi}{16}}$.
\end{proof}
Finally we need the following lemma, whose proof can be found in the cited paper.
\begin{lemma}\label{An-Bn} [\cite{ref1}, Prop. 4.3]
For large $n,$ we have
$$\mathbb{P}(A_n \backslash B_n\mid \mathcal{F}_{2n})=\mathcal{O}(n^{-\frac{1}{4}+\frac{2\delta_3+\delta_1}{2}}).$$
\end{lemma}

\begin{corollary}
$$\sum_{n\in\mathbb{N}} p_{n,2}<\infty.$$
\end{corollary}
\begin{proof}
\begin{align*}
 p_{n,2}=& \mathbb{P}(X_{2n}=0, Y_{2n}=0, A_n\backslash B_n)\\
 =&\mathbb{E}(\mathbf{1}_{Y_{2n=0}}\mathbb{E}(\mathbf{1}_{A_n\backslash B_n}\mathbf{1}_{X_{2n=0}} \mid \mathcal{F}_{2n}))\\
=&\mathbb{E}(\mathbf{1}_{Y_{2n=0}}\mathbb{E}(\mathbb{E}(\mathbf{1}_{A_n\backslash B_n}\mathbf{1}_{X_{2n=0}} \mid \mathcal{F}_{2n}\vee \mathcal{G})\mid \mathcal{F}_{2n}))\\
=&\mathbb{E}(\mathbf{1}_{Y_{2n=0}}\mathbb{E}(\mathbf{1}_{A_n\backslash B_n}\mathbb{P}(X_{2n}=0 \mid \mathcal{F}_{2n}\vee \mathcal{G})\mid \mathcal{F}_{2n}))\\
=& \mathcal{O}\left(n^{-\frac{1}{2}}n^{-\frac{1}{2}}n^{-\frac{1}{4}+\frac{2\delta_3+\delta_1}{2}}\right)
= \mathcal{O}(n^{-\frac{5}{4}+\frac{2\delta_3+\delta_1}{2}}).
\end{align*}
Where we used the estimates of lemma \ref{Y_n_behaviuour}, lemma \ref{Xn-lnn/n} and lemma \ref{An-Bn}. Now it's enough to choose $2\delta_3+\delta_1<\frac{1}{2}.$
\end{proof}

\subsubsection{Estimate of $p_{n,3}$}
Notice that $A_n^c=A_{n,1}^c \cup A_{n,2}^c.$ We are going to provide exponential estimates of both $\mathbb{P}(A_{n,1}^c \mid Y_{2n}=0)$ and $\mathbb{P}(A_{n,2}^c \mid Y_{2n}=0)$.
\begin{lemma}\label{lemma::lambda}
We have, for large $n$ and for every $t>0$
$$\mathbb{E}(e^{tY_{2n}})\sim c\left(q \cosh t + \sqrt{q^2 \cosh^2 t -(2q-1)}\right)^{2n},$$ where $c>0.$
\end{lemma}
\begin{proof}
We have, by the Markov property,
\begin{equation}\label{formula:: Markov}
 \mathbb{E}_{\nu_0}(e^{tY_{2n}})=e^{t\nu_0}\int \pi_{\nu}(\nu_0,dy_1)e^{ty_1}\int \pi_{\nu}(y_{1},dy_{2})e^{ty_{2}} \dots \int \pi_{\nu}(y_{2n-1},dy_{2n})e^{ty_{2n}}.
\end{equation}
It is now easy to see that can compute the quantity (\ref{formula:: Markov}) by means of the $2n$-th power of the matrix
\[ \pi_{\nu,t}:=\left( \begin{array}{cc}
q e^t & (1-q)e^{-t} \\
(1-q) e^t & qe^{-t} \end{array} \right)\]
which has the following eigenvalues
$$\lambda_{1,2}=q \cosh t \pm \sqrt{q^2 \cosh^2 t -(2q-1)}.$$
By the spectral decomposition, we know that $(\pi_{\nu,t})^{2n} \sim \lambda_1^{2n}(t) h_1h_1^T$ for large $n$, where $\lambda_1$ is the largest eigenvalue and $h_1$ represents the (column) eigenvector associated with $\lambda_1$. Hence for large $n$
$$\mathbb{E}(e^{tY_{2n}})=\sum_{y\in\{1,-1\}}(\pi_{\nu,t})^{2n} (\nu_0,y) \sim c\lambda_1^{2n}(t), \,\,\,\,\,c>0.$$
\end{proof}
\begin{proposition}\label{estimateA:n,1}
For large $n$, there exist $\delta>0$ such that
$$\mathbb{P}(A_{n,1}^c \mid Y_{2n}=0)= \mathcal{O} \left(\exp\left(-n^{\delta}\right)\right).$$
\end{proposition}
\begin{proof}
Let $a_n=[n^{\frac{1}{2}+\delta_1}]$; we have
\begin{align*}
\mathbb{P}(\max_{0\leq k\leq 2n} Y_k\geq a_n \mid Y_{2n}=0)=& \sum_{y\in\{a_n, a_{n+1},...,n\}} \frac{\mathbb{P}(\max_{0\leq k\leq 2n} Y_k=y, Y_{2n}=0)}{\mathbb{P}(Y_{2n}=0)}
\end{align*}
The estimate for $\mathbb{P}(\min_{0\leq k\leq 2n} Y_k\leq -a_n \mid Y_{2n}=0)$ can be obtained by the same argument, so we shall omit it.
By the reflection principle (note that the probability of any reflected path is equal to a multiplicative constant times the probability of the original path, this constant being $1/2$ or $2$)
\begin{align*}
\sum_{y\in\{a_n, a_{n+1},...,n\}}\mathbb{P}(\max_{0\leq k\leq 2n} Y_k=y, Y_{2n}=0)\leq& 2 \sum_{y\in\{a_n, a_{n+1},...,n\}}\mathbb{P}(Y_{2n}=2y)\\
=& 2\mathbb{P}(Y_{2n}\geq 2a_n)\\
\leq & 2 \inf_{t>0} \mathbb{P}(\exp(tY_{2n})\geq \exp(2ta_n))\\
\leq & 2 \inf_{t>0}\frac{\mathbb{E}(e^{tY_{2n}})}{e^{2ta_n}}.
\end{align*}
By Lemma \ref{lemma::lambda}, we have that for large $n$
\begin{align*}
\mathbb{E}(e^{tY_{2n}})
\sim &  c\left(q \cosh t +\sqrt{q^2 \cosh^2 t -(2q-1)}\right)^{2n}.
\end{align*}
Now by the Taylor expansion at $t=0$, and substituting $q=\frac{1}{3}$, we obtain
\begin{align*}
q \cosh t + \sqrt{q^2 \cosh^2 t -(2q-1)}=&\frac{2+\sqrt{7}}{6}+st^2+o(t^2)
< 1 +st^2,
\end{align*}
with $s=\frac{1}{6}+\frac{1}{3\sqrt{7}},$ where the inequality holds for $t\leq t^*$ for sufficiently small $t^*$. Hence 
\begin{align*}
\inf_{t>0}\frac{\mathbb{E}(e^{tY_{2n}})}{e^{2ta_n}}<c\inf_{t>0, t\leq t^*}  \exp(-2ta_n)\exp(2nst^2)=c\exp\left(-\frac{a_n^2}{2sn}\right)=c\exp\left(\frac{-n^{2\delta_1}}{2s}\right).
\end{align*}
where in the first equality we used the fact that the minimum is attained at  $t=\frac{a_n}{2ns}$, which goes to $0$ as $n$ tends to infinite.
Then, putting all together and using lemma \ref{Y_n_behaviuour}, we obtain
$$\mathbb{P}(A_{n,1}^c \mid Y_{2n}=0)= \mathcal{O} \left(n n^{\frac{1}{2}}\exp\left(\frac{-n^{2\delta_1}}{2s}\right)\right).$$
\end{proof}

\begin{lemma}\label{first-return-genf}
Let $\sigma_{a,a}$ the time of first return to state $a\in\mathbb{Z}\times\{-1,1\}$ of the Markov chain $(Y_n,\nu_n)$ starting at $a$. 
We have
$$\mathbb{E}(e^{-t\sigma_{a,a}})\sim\exp(-c\sqrt{t}),$$
with $c>0$ (i.e. $\lim_{t\to 0} \frac{\mathbb{E}(e^{-t\sigma_{a,a}})}{\exp(-c\sqrt{t})}=1$).
\end{lemma}
\begin{proof}
Let $p_{a,a}^{(n)}:=\mathbb{P}_a((Y_n,\nu_n)=a)$ and $G_{a,a}(s):=\sum_{k=0}^\infty p_{a,a}^{(k)} s^k.$ By lemma \ref{Y_n_behaviuour} we have
$p_{a,a}^{(n)} \sim \frac{C}{\sqrt{n}}$  as $n\to \infty$
with $C>0$.
This implies, by the Tauberian theorem (\cite{Feller2}, p.447, th.5), that there exists $C_1>0$ such that
$G_{a,a}(s)\sim\frac{C_1}{\sqrt{1-s}}$ as $s\to 1$. Then, using a standard result from the theory of Markov chain (e.g. cf. \cite{Woess}), we see that as $s\to 1$
$$\mathbb{E}(s^{\sigma_{a,a}})=1-\frac{1}{G_{a,a}(s)}\sim1-c\sqrt{1-s},$$
where $c=C_1^{-1}.$
Then, if we write $s=e^{-t}$, we have for $t\to 0$ 
$$\mathbb{E}(e^{-t\sigma_{a,a}})\sim 1-c\sqrt{1-e^{-t}}\sim 1-c\sqrt{t} \sim e^{-c\sqrt{t}}.$$
\end{proof}
\begin{proposition}\label{estimateA:n,2}
For large $n$ there exist $\delta'>0$ such that
$$\mathbb{P}(A_{n,2}^c \mid Y_{2n}=0)= \mathcal{O}\left(\exp(-n^{\delta'})\right).$$
\end{proposition}
\begin{proof}
We have
$$\mathbb{P}(A_{n,2}^c \mid Y_{2n}=0)=\mathbb{P}\left(\max_{y\in \mathbb{Z}}\eta_{2n-1}(y)\geq n^{\frac{1}{2} + \delta_2}\mid Y_{2n}=0\right)\leq \sum_{y\in\mathbb{Z}} \frac{\mathbb{P}\left(\eta_{2n-1}(y)\geq n^{\frac{1}{2} + \delta_2}\right)}{\mathbb{P}(Y_{2n}=0)}.$$
On the other hand we have
\begin{equation}\label{eq::1}
\mathbb{P}\left(\eta_{2n-1}(y)\geq a_n\right)\leq \mathbb{P}\left(\eta_{2n-1}(y,1)\geq \frac{a_n}{2}\right)+\mathbb{P}\left(\eta_{2n-1}(y,-1)\geq \frac{a_n}{2}\right).
\end{equation}
Now let $\sigma_{a,a}^{(k)}$ be the time of $k$-th return to point $a$ for the process $(Y_n,\nu_n)_{n\geq 0}$ starting at $a$. Observe that
$\mathbb{P}\left(\eta_{2n-1}(a)\geq a_n\right)\leq \mathbb{P}_a\left(\sigma_{a,a}^{([a_n])}\leq 2n \right)$
and consider the first term at the right-hand-side of (\ref{eq::1}). Notice that by lemma \ref{first-return-genf}, $\mathbb{E}(e^{-t\sigma_{a,a}})^m \sim \exp(-cm\sqrt{t})$ for every $m\in\mathbb{N}$; then, for $C>1$ there exists $t^*$ s.t. for every $t<t^*$, $\mathbb{E}(e^{-t\sigma_{a,a}})^m \leq C\exp(-cm\sqrt{t})$. Hence
for sufficiently large $n$
\begin{align*}
\mathbb{P}\left(\eta_{2n-1}(y,1)\geq \frac{a_n}{2}\right)\leq&  \inf_{t>0}\mathbb{P}_y \left(\exp\left(-t\sigma_{(y, 1),(y, 1)}^{([\frac{a_n}{2}])}\right)\geq \exp(-2nt)\right)\\
\leq & \inf_{t>0}\exp(2nt)\left(\mathbb{E}\left(\exp\left(-t\sigma_{(y, 1),(y, 1)}^{(1)}\right)\right)\right)^{[\frac{a_n}{2}]}\\
\leq & C\inf_{t>0, t<t^*} \exp\left(2nt-\frac{a_n}{2}c\sqrt{t}\right)\\
=& C\exp\left(-\frac{c^2a_n^2}{32n}\right)=C\exp\left(-c'n^{2\delta_2}\right)
\end{align*}
with $c'=\frac{c^2}{32}$, where we used the fact that the minimum is attained at $t=\left(\frac{ca_n}{8n}\right)^2$.

Since we can provide, with the same procedure, an exponential estimate also for $\eta_{2n-1}(y,-1)$, we finally obtain by lemma \ref{Y_n_behaviuour}
$$\mathbb{P}(A_{n,2}^c \mid Y_{2n}=0)\leq \sum_{y\in\mathbb{Z}} \frac{\mathbb{P}\left(\eta_{2n-1}(y)\geq n^{\frac{1}{2} + \delta_2}\right)}{\mathbb{P}(Y_{2n}=0)}= \mathcal{O}\left(nn^{\frac{1}{2}}\exp(-cn^{\delta_2})\right).$$
\end{proof}
\begin{corollary}
$$\sum_{n\in \mathbb{N}} p_{n,3}<\infty.$$
\end{corollary}
\begin{proof}
Combining proposition \ref{estimateA:n,1} and \ref{estimateA:n,2}, we know that for large $n$ $$\mathbb{P}(A_n^c\mid Y_{2n}=0)=\mathcal{O}(\exp(-n^{\min\{\delta,\delta'\}})).$$ Then the result follows by the trivial majorization
$$p_{n,3}:=\mathbb{P}( X_{2n}=0, Y_{2n}=0, A_n^c)\leq \mathbb{P}( Y_{2n}=0, A_n^c)\leq  \mathbb{P}(A_n^c\mid Y_{2n}=0).$$ 
\end{proof}
This completes the proof of theorem \ref{Th:: RANDOM}.

\subsection{The random walk on the lattice $\mathbf{H}_{f}$}
This section is devoted to the proof of theorem \ref{Th:: PERIODIC}.
Let $\mathbb{Z}_Q=\mathbb{Z}/ Q$ and, for every $y\in\mathbb{Z}$, we write $\overline y=y  \mod Q$. Define for every $n>0$
\begin{equation*}
W_n:=(Y_{n-1}, \nu_{n-1}; 	Y_{n}, \nu_{n})
\end{equation*}
and
\begin{equation*}
\overline W_n:=(\overline Y_{n-1}, \nu_{n-1}; 	\overline Y_{n}, \nu_{n}),
\end{equation*}
where $\overline Y_n=Y_n \mod Q$, and $W_0:=(-1,-1;0,1)$ , $\overline W_0:=(\overline{-1},-1;\overline{0},1)$. 
\begin{lemma}
	The process $(\overline W_n)_{n\geq 0}$ is a one-class recurrent Markov chain with period $2$. Its stationary distribution $\pi$ is defined as follows
	\begin{align}\label{stationary}
		\begin{cases}
			\pi(\overline y, \nu; \overline y', \nu')= \frac{2}{3}\frac{1}{2Q} \,\,\,\,\,\, \text{ if } \nu\ne \nu', \overline y'= \overline {y+\nu'}  \\
			\pi(\overline y, \nu; \overline y', \nu')= \frac{1}{3}\frac{1}{2Q} \,\,\,\,\,\, \text{ if } \nu=\nu', \overline y'= \overline {y+\nu'}
		\end{cases}
	\end{align}
\end{lemma}
\begin{proof}
	It is easy to verify that $(\overline Y_n, \nu_n)_{n\geq 0}$ is a Markov chain with $2Q$ states and period $2$, and that its stationary distribution is $\tilde \pi (\overline y,\nu)= \frac{1}{2Q}$, $\forall (\overline y,\nu)\in \mathbb{Z}_Q \times \{-1,1\}$. Then $(\overline Y_n, \nu_n; 	\overline Y_{n+1}, \nu_{n+1})_{n\geq 0}$ is again a MC, whose stationary distribution $\pi$ is directly derived from $\tilde\pi$ by defining
	$$\pi(\overline y, \nu; \overline y', \nu'):=\tilde \pi (\overline y, \nu) p_{(\overline y, \nu), (\overline y', \nu')},$$
	where $p_{\centerdot,\centerdot}$ is the transition probability of $(\overline Y_n, \nu_n)_{n\geq 0}$.
	The others statements are straightforward to verify.
\end{proof}


Note that $W_n$ enclose the information of the last three movements of the vertical skeleton $Y_n$: the reason for considering such a process, and its analogous $\overline W_n$ in $(\mathbb{Z}_Q\times\{-1,1\})^2$, is that we will need to control the number of times $(Y_n)_{n \geq 0}$ ``changes direction" at a certain level before it returns to the origin. This is done essentially by taking advantage of the periodicity of the orientations, and will in turn enable us to bound the difference between the number of
steps to the right and to the left of the embedded random walk $X_n$, distinguishing between the odd-valued and the even-valued steps, and to deduce that the probability of $X_n$ returning to $0$ is of order $n^{-1/2}$ for a set of paths with positive probability, which will imply the recurrence of $M$.

We begin by defining the following functionals of $\overline W.$
$$S_{n,e}:=\sum_{i=1}^{2n} f_e(\overline W_i):=\sum_{i=1}^{2n} f(\overline Y_{i-1})\mathbf{1}_{\{\nu_{i-1} \ne \nu_i\}},$$  
$$S_{n,o}:=\sum_{i=1}^{2n} f_o(\overline W_i):=\sum_{i=1}^{2n} f(\overline Y_{i-1})\mathbf{1}_{\{\nu_{i-1} = \nu_i\}}.$$

Moreover for every $n\in\mathbb{N}$ define the event
\begin{equation}\label{event}
\mathcal{Z}_n:= \{ W_{2n}= W_0\}=\{\overline W_{2n} = \overline W_0, Y_{2n}=0\}
\end{equation}

\begin{proposition}\label{lemma:: SeSoestimate}
Let $C>0.$ We have, for sufficiently large $n$
$$\mathbb{P}(|S_{n,e}|+|S_{n,o}|\leq C\sqrt{n}\mid \mathcal{Z}_n)\geq \delta_{C}>0.$$
\end{proposition}
\begin{proof}
To simplify our notation, we identify the states of $\overline W_n$ with the integers $\{1,2,...,4Q\}$, with arbitrary order. 
Accordingly we define $$\pi=(\pi_1,...,\pi_{4Q})$$ to be the vector where the $i$-th component is the value that the stationary distribution takes at state $i$, and the occupation measure  $$\overline\eta_n=(\overline\eta_n(1),...,\overline\eta_n(4Q)),$$ where $\overline\eta_n(i):=\sum_{k=0}^{n} \mathbf{1}_{\{\overline W_k=i\}},$ for $1\leq i\leq 4Q.$
By definition we have $$S_{n,e}=\sum_{i=1}^{4Q} u_i \overline\eta_{2n}(i)=u\overline \eta_{2n}^T,$$ where  $u\in\{-1,0,1\}^{4Q}$ is the vector such that $u_i$ equals to the value that $f_e$ takes on the $i$-th state. 
Analogously, let $v\in\{-1,0,1\}^{4Q}$ such that  $S_{n,o}=v\overline \eta_{2n}^T$ and $w\in\{-1,1\}^{4Q}$ such that $Y_{2n}=\sum_{i=1}^{2n}\nu_i = w\overline \eta_{2n}^T$ . 
Note that $u,v,w$ are linearly independent vectors and that we have 
\begin{equation}\label{average}
u(2n\pi)^T=v(2n\pi)^T=w(2n\pi)^T=0,
\end{equation} by (\ref{stationary}).

Let $c>0$. By the multidimensional local limit theorem for the random vector $\overline \eta_{2n}$  (lemma 16 in \cite{refkolmogorov}) we know that there exist a lattice $Z\subset \mathbb{Z}^{4Q}$ of dimension $r$, $2<r\leq4Q,$ and a constant $c'>0$ dependent of $c$, such that
\begin{equation}\label{lowerBound}
\mathbb{P}(\overline\eta_{2n}= x, \overline W_{2n}=\overline W_{0})  \geq  \frac{c'}{n^{r/2}},
\end{equation}
for large $n$ and for all $x\in Z$ such that $|x_i-\pi_i| \leq c\sqrt{n}$, $1\leq i\leq 4Q$.
Hence, by (\ref{lowerBound}) and (\ref{average}), and taking $c=\frac{C}{4Q},$ we have
\begin{align}\label{lasteq}
\mathbb{P}(|S_{n,e}|+|S_{n,o}|\leq C\sqrt{n}; \mathcal{Z}_n) 
\geq & \mathbb{P}\left(|\overline \eta_{2n}(i) - \pi_i|\leq c\sqrt{n}, \forall i; \mathcal{Z}_n\right)\nonumber\\
=&\sum_{\substack{x\in Z,  w x^T=0, \\|x_i-\pi_i| \leq c\sqrt{n}, \forall i}} \mathbb{P}(\overline \eta_{2n}=x, \overline W_{2n}=\overline W_{0})\nonumber \\
\geq & |\{x\in Z,  w x^T=0, |x_i-\pi_i| \leq c\sqrt{n}, \forall i\}| \frac{c'}{n^{r/2}}\nonumber\\
=& C'\frac{n^{(r-1)/2)}}{n^{r/2}}\geq \frac{C'}{\sqrt{n}},
\end{align}
with $C'>0$.
Finally by (\ref{lasteq}) and lemma \ref{Y_n_behaviuour} $$\mathbb{P}(|S_{n,e}|+|S_{n,o}|\leq C\sqrt{n}| \mathcal{Z}_n)\geq \frac{\mathbb{P}(|S_{n,e}|+|S_{n,o}|\leq C\sqrt{n}, \mathcal{Z}_n)}{\mathbb{P}(Y_{2n}=0)}\geq \delta_C>0.$$
%
\end{proof}
\subsubsection{Proof of recurrence} 
Define the following set of constrained paths 
\begin{align*}
\mathrm{Constr}(n,f):=& \left\{(\gamma,q): \{-1,0,1,...,2n\} \longrightarrow \mathbb{Z}\times \{-1,1\} \text{ s.t. } \forall i, \, \gamma(i)=  \gamma(i-1)\pm 1, \right. \\
&\left. (\gamma(-1), q(-1);\gamma(0), q(0))=(\gamma(2n-1), q(2n-1); \gamma(2n),q(2n))=W_0, \right. \\
& \left. \left|\sum_{i=1}^{2n} f(\overline \gamma_{i-1})\mathbf{1}_{\{q_{i-1} \ne q_i\}}\right| +\left|\sum_{i=1}^{2n} f(\overline \gamma_{i-1})\mathbf{1}_{\{q_{i-1} = q_i\}}\right|\leq C\sqrt{n}\right\}.
\end{align*}
Observe that if we prove that for every  $(\gamma,q)\in \mathrm{Constr}(n,f)$
\begin{equation}\label{eq:: Xn-keyestimate}
\mathbb{P}(X_{2n}=0 \mid (Y_i,\nu_i)=(\gamma (i), q(i))\, \forall i\leq 2n)\geq \frac{c}{\sqrt{n}},
\end{equation}
 then the recurrence of the random walk will follow: in fact, thanks to (\ref{eq:: Xn-keyestimate}) and to Proposition \ref{lemma:: SeSoestimate} we'd have for large $n$
\begin{align*}
\mathbb{P}(X_{2n}=0, Y_{2n}=0)\geq& \mathbb{P}\left(X_{2n}=0,\mathcal{Z}_n, |S_{n,e}|+|S_{n,o}|\leq C\sqrt{n} \right)\\
=& \sum_{(\gamma,q) \in \mathrm{Constr}(n,f)} \mathbb{P}(X_{2n}=0 \mid (Y_i,\nu_i)=(\gamma (i), q(i))\, \forall i\leq 2n)\\
\times & \mathbb{P}((Y_i,\nu_i)=(\gamma (i), q(i))\, \forall i\leq 2n)\\
\geq&\frac{c}{\sqrt{n}}\sum_{(\gamma,q) \in \mathrm{Constr}(n,f)} \mathbb{P}((Y_i,\nu_i)=(\gamma (i), q(i))\, \forall i\leq 2n)\\
=& \frac{c}{\sqrt{n}}\mathbb{P}(|S_{n,e}|+|S_{n,o}|\leq C\sqrt{n}, \mathcal{Z}_n)\\
\geq& \frac{c'}{n},
\end{align*}
with $c'>0$.

From now on we fix $(\gamma,q)\in \mathrm{Constr}(n,f)$ and every probability will be taken conditionally to $$\{(Y_i,\nu_i)=(\gamma (i), q(i))\, \forall i\leq 2n\},$$ although, in order to simplify the notation, we will sometimes omit to write it. Let $N_e^+$ and $N_e^-$ be, respectively, the number of right (left) directed even steps of the embedded random walk up to time $2n$, and $N_o^+, N_o^-$ the analogous quantities for the odd steps. Observe that $S_{n,e}=N_e^+ - N_e^-$ and $S_{n,o}=N_o^+ - N_o^-$. In particular, since $(\gamma,q)\in \mathrm{Constr}(n,f)$, we have 
$$|N_e^+-N_e^-|+|N_o^+-N_o^-|\leq C\sqrt{n}.$$
\begin{lemma}
Conditionally to $\{(Y_i,\nu_i)=(\gamma (i), q(i))\, \forall i\leq 2n\}$, we have
$\mathbb{E}(X_{2n})=\mathcal{O}(\sqrt{n})$
and
$\sigma^2(X_{2n})\sim Cn$, $C>0.$
\end{lemma}
\begin{proof}
Note that
\begin{align*}
\mathbb{E}(X_{2n})
=& \mathbb{E}(\xi_{i,e})(N_e^+-N_e^-) + \mathbb{E}(\xi_{i,o})(N_o^+-N_o^-)
\leq \max\{\mathbb{E}(\xi_{i,o}),\mathbb{E}(\xi_{i,e})\} C\sqrt{n}. 
\end{align*} 
On the other hand, the variance of $X_{2n}$ is, by independence, the sum of the variances of the even and odd geometric random variables, and so since both of them has finite variance we obtain  
$\sigma^2(X_{2n})\sim C n$ for some $C>0$.
\end{proof}

\begin{lemma}\label{LLT-Xn}
There exists $c>0$ such that, for every large $n$ and conditionally to $\{(Y_i,\nu_i)=(\gamma (i), q(i))\, \forall i\leq 2n\}$
we have 
\begin{equation}\label{eq::estimateXn}
\mathbb{P}(X_{2n}=0)\geq \frac{c}{\sqrt{n}}.
\end{equation} 
\end{lemma}
\begin{proof}
For every $k\in\mathbb{N}$ let $\xi_k$ be the random variable that represents the $k$-th step of the horizontal random walk $X_n$. We write
\begin{align*}
X_{2n}=\sum_{k=1}^{2n} \xi_k=&\sum_{i=1}^{N_e^+}\xi_{i,e}+\sum_{i=1}^{N_o^+}\xi_{i,o}-\sum_{i=N_e^+ +1}^{N_e^+ +N_e^-}\xi_{i,e}-\sum_{i=N_o^+ +1}^{N_o^+ +N_o^-}\xi_{i,o},
\end{align*}  
and for every $k$ let $a_k:=\mathbb{E}(\xi_k)$, $b_k^2:=\sigma^2(\xi_k)$ and $$A_n:=\sum_{i=1}^{2n} \mathbb{E}(\xi_k), \,\,\,\,\,\,\,\, B_n^2:= \sum_{i=1}^{2n} \sigma^2(\xi_k).$$ 

First, we are going to show that
\begin{equation}\label{eq::LLTxn}
\left| B_n\mathbb{P}(X_{2n} =0)- \frac{2}{\sqrt{2\pi }} e^{-\frac{1}{2} \left(\frac{A_n}{B_n}\right)^2} \right|\to 0,
\end{equation} 
as $n\to \infty$. \footnote{Notice that, if $(Y_{2n}, \nu_{2n})=(0,1)$, which happens in our case since we are considering a constrained path satisfying this property, then the value taken by $X_{2n}$ is either a null or a even integer (cf. figure \ref{figure::honeycomb}); in particular, $\mathbb{P}(X_{2n} =0)>0$. 
}
Then, thanks to the previous lemma, we obtain the estimate (\ref{eq::estimateXn}).

To prove (\ref{eq::LLTxn}), we generalize a classical approach (cf. \cite{Gnedenko}). Let $\phi_{\xi_k}(t)=\mathbb{E}(e^{it\xi_k}),$ and $\phi_{X_{2n}}(t)=\mathbb{E}(e^{it X_{2n}})=\phi_{\sum_{k=1}^{2n} \xi_k}(t)=\Pi_{k=1}^{2n} \phi_{\xi_k}(t)$, and precisely
\begin{align*}
\phi_{X_{2n}}(t)
=& \chi_e(t)^{N_e^+}\chi_o(t)^{N_o^+}\chi_e(-t)^{N_e^-}\chi_o(-t)^{N_o^-},
\end{align*} 
where we recall that
$\chi_o(t)=\frac{3e^{it}}{4-e^{2it}}$
and $\chi_e(t)=e^{-it} \chi_o(t).$
In particular note that $|\chi_o(t)|=|\chi_e(t)|=|\chi_o(-t)|=|\chi_e(-t)|=1$ for $t=0$ and $t=\pi$, and $<1$ otherwise.
Now, since $\sum_{k=-\infty}^\infty \mathbb{P}(X_{2n}=2k) e^{i2kt}=\phi_{X_{2n}}(t)$, if we integrate both sides of this equation from $-\pi/2$ to $\pi/2$ we obtain $\pi\mathbb{P}(X_{2n}=0)=\int_{\frac{-\pi}{2}}^{\frac{\pi}{2}} \phi_{X_{2n}}(x) dx.$
Then
$$\pi\mathbb{P}(X_{2n}=0)=\frac{1}{B_n}\int_{\frac{-\pi B_n}{2}}^{\frac{\pi B_n}{2}} \phi_{X_{2n}}(t/B_n) dt=\frac{1}{B_n}\int_{\frac{-\pi B_n}{2}}^{\frac{\pi B_n}{2}} e^{it\frac{A_n}{B_n}}\phi_{\frac{X_n-A_n}{B_n}}(t) dt.$$
The following equality is easily proved for every $z\in\mathbb{R}$. 
$$\frac{1}{\sqrt{2\pi}}e^{-\frac{1}{2}z^2} =\frac{1}{2\pi} \int e^{-itz-\frac{t^2}{2}} dt.$$ 
In particular, in our case, we take $z:=-\frac{A_n}{B_n}. $
We write 
\begin{equation}
R_n:=2\pi\left[\frac{B_n}{2} \mathbb{P}(X_{2n}=0)-\frac{1}{\sqrt{2\pi}}e^{-\frac{1}{2}(\frac{A_n}{B_n})^2}\right]=J_1+J_2+J_3+J_4,
\end{equation}
where
\begin{align*}
J_1=&\int_{-A}^A e^{it\frac{A_n}{B_n}} \left[\phi_{\frac{X_{2n}-A_n}{B_n}}(t) - e^{-\frac{t^2}{2}} \right] dt\\
J_2=& - \int_{|t|>A} e^{it\frac{A_n}{B_n}-\frac{t^2}{2}} dt\\
J_3=& \int_{\epsilon B_n<|t|<\pi B_n/2} e^{it\frac{A_n}{B_n}} \phi_{\frac{X_{2n}-A_n}{B_n}}(t) dt\\
J_4=& \int_{A<|t|<\epsilon B_n} e^{it\frac{A_n}{B_n}} \phi_{\frac{X_{2n}-A_n}{B_n}}(t) dt
\end{align*}
So to complete the proof we must show that these quantities tend to $0$ as $n\to \infty$ and for sufficiently large $A$ and small $\epsilon$.

First, we show that the sequence $(\xi_k)_{k\geq 1}$ satisfies the Lyapunov condition with $\delta=1$, that is
$$\lim_{n\to \infty} \frac{1}{B_n^{2+\delta}} \sum_{k=1}^{2n} \mathbb{E}|\xi_k - a_k|^{2+\delta}=0.$$
In fact, by the previous lemma, $B_n^2 \sim Cn$ with $C>0$ and the $\xi_k$'s clearly have finite moment of the third order, so for appropriate $C'>0$
$$\frac{1}{B_n^{3}} \sum_{k=1}^{2n} \mathbb{E}|\xi_k - a_k|^{3}\sim\frac{1}{Cn^{3/2}} \sum_{k=1}^{2n} \mathbb{E}|\xi_k - a_k|^{3} \leq \frac{C'n}{n^{3/2}} \sim \frac{C'}{n^{1/2}}.$$
 Then by the CLT we have that, as $n\to \infty$, 
$$\phi_{\frac{X_{2n}-A_n}{B_n}}(t)\to e^{-\frac{t^2}{2}}$$
which implies $|J_1| \to 0.$

We have $$|J_2|\leq \int_{|t|>A} |e^{-it\frac{A_n}{B_n}}| |e^{-\frac{t^2}{2}}| dt = \int_{|t|>A} |e^{-\frac{t^2}{2}}| \leq \frac{2}{A}e^{-\frac{A^2}{2}}$$
and so by choosing a sufficiently large $A$ we can make $J_2$ arbitrarily small.

For every $k$, $\phi_{\xi_k}(t)$ is either $\chi_e(t)$, $\chi_e(-t)$, $\chi_o(t)$ or $\chi_o(-t)$. Since for $\epsilon < |t| < \pi/2$ we have $|\phi_{\xi_k}(t)|<1,$ we can find $c>0$ such that $|\phi_{\xi_k}(t)|\leq e^{-c}<1$ for every $k$. Then, if $\epsilon B_n < |t| < \pi B_n/2$, we have
\begin{align*}
|\phi_{\frac{X_{2n}-A_n}{B_n}}(t)|=& \Pi_{k=1}^{2n} |\phi_{\xi_k - a_k}(t/B_n)|
= \Pi_{k=1}^{2n} |e^{-ia_k t/B_n}||\phi_{\xi_k }(t/B_n)|\\
=&\Pi_{k=1}^{2n} |\phi_{\xi_k }(t/B_n)|\leq \Pi_{k=1}^{2n} e^{-c} = e^{-cn}, 
\end{align*}
which tends to $0$ as $n\to \infty$. This implies $|J_3|\to 0$ as $n\to \infty$.

By the Taylor expansion at $t=0$
\begin{align*}
|\phi_{\frac{X_{2n}-A_n}{B_n}}(t)|
=&\Pi_{k=1}^{2n} |\phi_{\xi_k - a_k}(t/B_n)| = \Pi_{k=1}^{2n} \left|1- \frac{\sigma_k^2t^2}{2B_n^2}+o\left(\frac{t^2}{B_n^2}\right)\right|.
\end{align*}
Now, if $|t|\leq \epsilon B_n$ for sufficiently small $\epsilon$, we have 
\begin{align*}
|\phi_{\frac{X_{2n}-A_n}{B_n}}(t)|
<& \Pi_{k=1}^{2n} \left|1-  \frac{\sigma_k^2t^2}{4B_n^2}\right|< \Pi_{k=1}^{2n} e^{-\frac{\sigma_k^2t^2}{4B_n^2}} = e^{-t^2/4}.
\end{align*}
Then $$|J_4|\leq 2\int_A^{\epsilon B_n}e^{-t^2/4} dt < 2\int_A^\infty e^{-t^2/4} dt$$ where the right hand side tends to $0$ as $A \to \infty$. So we can make $|J_4|$ arbitrarily small.
\end{proof}
The proof of recurrence is now complete.

\subsection{The random walk on the $\mathbf{H}_{\overline \epsilon,\lambda}$ lattice}
This section is devoted to the proof of theorem \ref{Th:: PERTURBED}.

\subsubsection{Proof of theorem \ref{Th:: PERTURBED} (i)} To prove a.s. transience, we can follow the same technique we used for the case of a random environment defining, for $n\geq 0$, the events $A_n$ and $B_n$ just as before (the only difference is that, this time, we write $\overline \epsilon y$ in place of $\epsilon_y$).
Now, it is clear that many of the estimates we obtained in the case of a random environment still hold: in fact, according to \cite{ref2}, we only need to provide an estimate on $A_n\backslash B_n$, conditionally to $\mathcal{F}:=\sigma((Y_i,\nu_i); n=1,...,n)$. This estimate is given by the following result, whose proof can be found in the cited paper.
\begin{proposition}[Proposition 3.2, \cite{ref2}]
For all $\beta <1$, there exists a $\delta_\beta>0$ such that -uniformly in $\mathcal{F}$- for all large $n$
$$\mathbb{P}(A_n\backslash B_n\mid \mathcal{F} )=\mathcal{O}(n^{-\delta_\beta}).$$
\end{proposition}

Then, exactly as in the case of random environment, we use the estimates to show that $\mathbb{P}(X_{2n}=0,Y_{2n}=0)$ is summable. This proves the a.s. transience. 

\subsubsection{Proof of theorem \ref{Th:: PERTURBED} (ii)}
To prove a.s. recurrence we need to show that
$\mathbb{P}(X_{2n}=0, Y_{2n}=0\mid \mathcal{G})=\infty,$
where $\mathcal{G}:=\sigma(\overline \epsilon_y, y\in\mathbb{Z}).$
We know from Borel-Cantelli lemma that for almost every realization of the environment, we have only a finite number of randomly perturbed directions around the origin. So, in what follows fix a realization $\overline \epsilon$ such that the number of perturbations is $L<\infty$; we will compute all the probabilities conditionally to $\overline \epsilon$, although we will not always specify that.

Let 
$$\overline S_{n,e}^{\leq L}:=\sum_{i=1}^{2n} \mathbf{1}_{\{\nu_{i-1}\ne \nu_i, |Y_i|\leq L\}} \overline\epsilon_{Y_i},$$
$$\overline S_{n,e}^{\geq L}:=\sum_{i=1}^{2n} \mathbf{1}_{\{\nu_{i-1}\ne \nu_i, |Y_i|\geq L\}} \overline\epsilon_{Y_i},$$
$$ S_{n,e}^{\leq L}:=\sum_{i=1}^{2n} \mathbf{1}_{\{\nu_{i-1}\ne \nu_i, |Y_i|\leq L\}}  f(\overline Y_i),$$
$$S_{n,e}^{\geq L}:=\sum_{i=1}^{2n} \mathbf{1}_{\{\nu_{i-1}\ne \nu_i, |Y_i|\geq L\}} f(\overline Y_i).$$
Note that $S_{n,e}^{\geq L}=\overline S_{n,e}^{\geq L}$. Moreover let  $$\overline S_{n,e}=\overline S_{n,e}^{\leq L} + \overline S_{n,e}^{\geq L},$$
$$S_{n,e}=S_{n,e}^{\leq L} + S_{n,e}^{\geq L}.$$
In a completely analogous way we define the quantities corresponding to the odd steps: $S_{n,o}^{\leq L}, S_{n,o}^{\geq L}, \overline S_{n,o}^{\leq L}, \overline S_{n,o}^{\geq L},\overline S_{n,o}, S_{n,o}.$
\begin{lemma}\label{lemma:: L}
We have 
$$|\overline S_{n,e}| \leq 2\sum_{i=1}^{2n} \mathbf{1}_{\{|Y_i|\leq  L\}} + | S_{n,e}|,$$
$$|\overline S_{n,o}| \leq 2\sum_{i=1}^{2n} \mathbf{1}_{\{|Y_i|\leq  L\}} + | S_{n,o}|.$$
\end{lemma} 
\begin{proof}
We have
\begin{align*}
|\overline S_{n,e}|
= & |\overline S_{n,e}^{\leq L} + \overline S_{n,e}^{\geq L}|\\
= & |\overline S_{n,e}^{\leq L} - S_{n,e}^{\leq L} + S_{n,e}^{\leq L}+ \overline S_{n,e}^{\geq L}|\\
= & |\overline S_{n,e}^{\leq L} -  S_{n,e}^{\leq L} +  S_{n,e}|\\
\leq & |\overline S_{n,e}^{\leq L} -  S_{n,e}^{\leq L}| + | S_{n,e}|\\
\leq & 2\sum_{i=1}^{2n} \mathbf{1}_{\{|Y_i|\leq L\}} + | S_{n,e}|.
\end{align*}
The same argument proves the analogous majorization for $\overline S_{n,o}$.
\end{proof}

We shall denote again by $\mathcal{Z}_n$ to the event (\ref{event}).
\begin{lemma} \label{corollary:bound}
We can find $c'>0$ such that for every $n\in\mathbb{N}$
$$\mathbb{E}\left(\sum_{i=1}^{2n} \mathbf{1}_{\{|Y_i|\leq L\}}\mid \mathcal{Z}_n\right)\leq c'\sqrt{n}.$$
\end{lemma}
\begin{proof}
 We have
 \begin{align*}
\mathbb{E}\left(\sum_{i=1}^{2n} \mathbf{1}_{\{|Y_i|\leq  L\}}\mid \mathcal{Z}_n\right)=& \sum_{i=1}^{2n} \mathbb{P}(|Y_i|\leq  L \mid \mathcal{Z}_n)
= \sum_{k=- L}^{L}\sum_{i=1}^{2n} \mathbb{P}(|Y_i|=k	\mid \mathcal{Z}_n).
 \end{align*}
Again by the LLT for Markov chains (\cite{refkolmogorov}), we have that $\mathbb{P}_0(Y_{i}=k)$ is majorized by $\frac{c}{\sqrt{i}}$ for an appropriate constant $c>0$ independent of $k$ and for all sufficiently large $i$; Then can find $c'>0$ large enough such that $\mathbb{P}_0(Y_{i}=k)\leq\frac{c'}{\sqrt{i}}$ for all $i>0$. Hence
\begin{align*}
\sum_{i=1}^{2n} \mathbb{P}(Y_i=k	\mid \mathcal{Z}_n)
\leq & \frac{\sum_{i=1}^{2n} \mathbb{P}_0(Y_i=k)\mathbb{P}_{k}(Y_{2n-i}=0)}{\mathbb{P}_0(\mathcal{Z}_n)}\\
\leq& C\sqrt{n}\int_{t=0}^{2n}\frac{1}{\sqrt{t(2n-t)}} dt
= C\sqrt{n}\left[\arcsin\left(\frac{t-2n}{2n}\right)\right]_0^{2n} \leq c'\sqrt{n}.
\end{align*}
\end{proof}
\begin{corollary}\label{cor:: SeSo}
We have
$$\mathbb{P}(|\overline S_{n,e}| + |\overline S_{n,o}|\leq C \sqrt{n}\mid \mathcal{Z}_n)\geq K_{C,L}>0$$
with $C>0$ and sufficiently large $n$.
\end{corollary}
\begin{proof}
By lemma \ref{lemma:: L} we have for large $n$
\begin{align*}
\mathbb{P}\left(\frac{|\overline S_{n,e}| + |\overline S_{n,o}|}{\sqrt{n}} \leq  C \mid \mathcal{Z}_n\right)\geq&\mathbb{P}\left(\frac{4\sum_{i=1}^{2n} \mathbf{1}_{\{|Y_i|\leq L\}} + | S_{n,e}| + | S_{n,o}|}{\sqrt{n}} \leq  C \mid \mathcal{Z}_n\right)\\
\geq& \mathbb{P}\left(\frac{| S_{n,e}| + | S_{n,o}|}{\sqrt{n}} \leq  C/2, \frac{\sum_{i=1}^{2n} \mathbf{1}_{\{|Y_i|\leq L\}}}{\sqrt{n}}\leq C/2  \mid \mathcal{Z}_n\right)
\end{align*}
Now, by proposition \ref{lemma:: SeSoestimate} $$\mathbb{P}\left(\frac{| S_{n,e}| + | S_{n,o}|}{\sqrt{n}} \leq  C/2 \mid \mathcal{Z}_n\right)\geq \delta_{C}>0$$ 
and by the Markov inequality together with lemma \ref{corollary:bound} $$\mathbb{P}\left(\frac{\sum_{i=1}^{2n} \mathbf{1}_{\{|Y_i|\leq L\}}}{\sqrt{n}}\leq C/2 \mid \mathcal{Z}_n\right)\geq \delta_{C,L}'>0,$$ 
where both $\delta_C$ and $\delta_{C,L}'$ tend to $1$ as $C$ grows to infinity. So if we take a sufficiently large $C$ s.t. $\delta_{C',L}>1-\delta_C$, the intersection between these two events will have positive probability.
\end{proof}
Now, following the same argument used in the proof of theorem \ref{Th:: PERIODIC} (we shall not repeat it), one shows recurrence for the random walk conditionally to the environment $\overline\epsilon$. But since the choice of $\overline\epsilon$ is arbitrary, with the only requirement that there are only a finite number of perturbations around the origin, and since this requirement is satisfied by a.e. realization, the proof of a.s. recurrence is complete.

\section{Conclusion}
This paper shows that the random walk has the same recurrence behaviour as in the square grid lattice case.
It would be desirable to extend our results to a more general class of planar graphs with some undirected and some
directed bonds. We are confident that the techniques developed here may be useful for obtaining results on recurrence
in a more general setting.

\bibliographystyle{alea3}
\bibliography{BosiCampaninoArticle}

\begin{thebibliography}{11}
\providecommand{\natexlab}[1]{#1}
\providecommand{\url}[1]{\texttt{#1}}
\providecommand{\urlprefix}{URL }
\expandafter\ifx\csname urlstyle\endcsname\relax
  \providecommand{\doi}[1]{doi:\discretionary{}{}{}#1}\else
  \providecommand{\doi}{doi:\discretionary{}{}{}\begingroup
  \urlstyle{rm}\Url}\fi
\providecommand{\eprint}[2][]{\url{#2}}

\bibitem[{Campanino and Petritis(2003)}]{ref1}
M.~Campanino and D.~Petritis.
\newblock Random walks on randomly oriented lattices.
\newblock \emph{Mark. Proc. Rel. Fields} \textbf{9}, 391--412 (2003).

\bibitem[{Campanino and Petritis(2014)}]{ref2}
M.~Campanino and D.~Petritis.
\newblock Type transition of simple random walks on randomly directed regular
  lattices.
\newblock \emph{J. Appl. Prob.} \textbf{51}, 1065--1080 (2014).

\bibitem[{Castell et~al.(2011)Castell, Guillotin-Plantard, P{\`e}ne and
  Shapira}]{guillotin4}
F.~Castell, N.~Guillotin-Plantard, F.~P{\`e}ne and B.~Shapira.
\newblock A local limit theorem for random walks in random scenery and on
  randomly oriented lattices.
\newblock \emph{The Annals of Probability} \textbf{39}, 2079--2118 (2011).

\bibitem[{Devulder and P{\`e}ne(2013)}]{pene2}
A.~Devulder and F.~P{\`e}ne.
\newblock Random walk in random environment in a two-dimensional stratified
  medium with orientations.
\newblock \emph{Electron. J. Probab.} \textbf{18}, 1--23 (2013).

\bibitem[{Feller(1966)}]{Feller2}
W.~Feller.
\newblock \emph{An Introduction to Probability Theory and Its Applications}.
\newblock John Wiley \& Sons, New York. (1966).

\bibitem[{Gnedenko(1962)}]{Gnedenko}
B.V. Gnedenko.
\newblock \emph{The Theory of Probability}.
\newblock Chelsea Publ. Comp., New York. (1962).

\bibitem[{Guillotin-Plantard and Le~Ny(2008)}]{guillotin3}
N.~Guillotin-Plantard and A.~Le~Ny.
\newblock A functional limit theorem for a {2D}-random walk with dependent
  marginals.
\newblock \emph{Elec. Comm. in Probab.} \textbf{13}, 337--351 (2008).

\bibitem[{Guillotin-Plantard and P{\`e}ne(2015)}]{guillotin5}
N.~Guillotin-Plantard and F.~P{\`e}ne.
\newblock On the range of {Campanino and Petritis} random walk.
\newblock \emph{hal-01183813}  (2015).

\bibitem[{Kolmogorov(1949)}]{refkolmogorov}
A.N. Kolmogorov.
\newblock A local limit theorem for classical {Markov} chains.
\newblock \emph{Izv. Akad. Nauk SSSR Sero Mat.} \textbf{13}, 281--300 (1949).

\bibitem[{P{\`e}ne(2009)}]{pene}
F.~P{\`e}ne.
\newblock Transient random walk in {$Z^2$} with stationary orientations.
\newblock \emph{ESAIM: Probab. and Stat.} \textbf{13}, 417--436 (2009).

\bibitem[{Woess(2009)}]{Woess}
W.~Woess.
\newblock \emph{Denumerable {Markov} Chains}.
\newblock Eur. Math. Soc., Zurich (2009).

\end{thebibliography}

\end{document}